\documentclass{amsart}
\usepackage{amssymb}
\usepackage{amsmath}
\usepackage{amsthm}
\usepackage[all]{xy}
\usepackage{graphicx}
\usepackage{epstopdf}
\xyoption{2cell}

\newcommand{\Ol}{{\mathcal O}}
\newcommand{\f}{\varphi}
\newcommand{\oz}{\overline{Z}}
\newcommand{\vt}{\vartheta}
\newcommand{\ac}{{\mathcal H}}
\newcommand{\E}{{\mathcal E}}
\newcommand{\hD}{{\widehat D}}
\newcommand{\V}{{\mathcal V}}
\newcommand{\W}{{\mathcal W}}
\newcommand{\pu}{{\mathbb P^1}}
\newcommand{\proj}{\mathbb P}
\newcommand{\R}{\mathbb R}
\newcommand{\quadr}{\mathbb Q}
\newcommand{\pd}{{\mathbb P^2}}

\newcommand{\tl}{\widetilde}
\newcommand{\Dcap}{\widehat D}

\DeclareMathOperator{\loc}{\mathrm{Locus}}
\DeclareMathOperator{\cloc}{\mathrm{ChLocus}}
\newcommand{\ratcurves}{\textrm{Ratcurves}^n(X)}

\newcommand{\Ch}{\textrm{Chow}}
\DeclareMathOperator{\cycl}{N_1}

\newcommand{\conx}[1]{\cone\,(#1,X)}
\newcommand{\cycx}[1]{\cycl(#1,X)}
\DeclareMathOperator{\cone}{NE}

\DeclareMathOperator{\pic}{Pic}

\DeclareMathOperator{\Exc}{Exc}

\newcommand{\rc}[2]{#1 \xymatrix{\ar@{-->}[r] & }{#2}}

\begin{document}

\newtheorem{theorem}{Theorem}[section]
\newtheorem*{theorem*}{Theorem}
\newtheorem{lemma}[theorem]{Lemma}
\newtheorem{proposition}[theorem]{Proposition}
\newtheorem{corollary}[theorem]{Corollary}
\newtheorem{example}[theorem]{Example}

\newtheorem{teo}{Theorem}[subsection]
\newtheorem{lem}[teo]{Lemma}
\newtheorem{prop}[teo]{Proposition}
\newtheorem{cor}[teo]{Corollary}

\theoremstyle{definition}
\newtheorem{definition}[theorem]{Definition}
\newtheorem{statement}[theorem]{}
\theoremstyle{remark}
\newtheorem{remark}[theorem]{Remark}

\theoremstyle{definition}
\newtheorem{defi}[teo]{Definition}
\newtheorem{ex}[teo]{Example}
\theoremstyle{remark}
\newtheorem{rmk}[teo]{Remark}

\renewcommand{\theequation}{{\arabic{section}.\arabic{theorem}.\arabic{equation}}}

\author{Carla Novelli}
\author{Gianluca Occhetta}

\address{Dipartimento di Matematica ``F. Casorati'',\newline Universit\`a di Pavia,\newline via Ferrata 1,\newline I-27100 Pavia}
\curraddr{Dipartimento di Matematica e Applicazioni, \newline Universit\`a di Milano - Bicocca,\newline via R. Cozzi 53, \newline I-20126 Milano}
\email{carla.novelli@unimib.it}

\address{Dipartimento di Matematica,\newline Universit\`a di Trento, \newline via Sommarive 14,\newline I-38050 Povo (TN)}
\email{gianluca.occhetta@unitn.it}

\title{Manifolds covered by lines and extremal rays}

\begin{abstract}
Let $X$ be a smooth complex projective variety and let $H \in \pic(X)$ be an ample line bundle.
Assume that $X$ is covered by rational curves with degree one with respect to $H$
and with anticanonical degree greater than or equal to $(\dim X -1)/2$.
We prove that there is a covering family of such curves whose numerical
class spans an extremal ray in the cone of curves $\cone(X)$.
\end{abstract}

\maketitle

%%%%%%%%%%%%%%%%%%%%%%%%%%%%%%%%%%%%%%%%%%%%%%%%%%%%
%%%%%%%%%%%%%%%%%%%                SECTION 0             %%%%%%%%%%%%%%%%%%%
%%%%%%%%%%%%%%%%%%%%%%%%%%%%%%%%%%%%%%%%%%%%%%%%%%%%

\section*{Introduction}

Let $X$ be a smooth complex projective variety which admits a morphism with connected fibers $\f \colon X \to Z$
onto a normal variety $Z$ such that the anticanonical bundle $-K_X$ is $\f$-ample, $\dim X > \dim Z$ and  
$\rho_X=\rho_Z+1$ ({\em i.e.} an elementary extremal contraction of fiber type).\\
It is well known, by fundamental results of Mori theory, that through every point of $X$
there is a rational curve contracted by $\f$. The numerical classes of these curves
lie in an extremal ray of the cone $\cone(X)$. By taking a covering family of such curves one obtains
a {\em quasi-unsplit} family of rational curves, {\em i.e.} a family such that
the irreducible components of all the degenerations of curves in the family are numerically
proportional to a curve in the family.
It is very natural to ask if the converse is also true:\par
\bigskip
\begin{center}
\begin{minipage}[center]{10cm}
Given a covering 
quasi-unsplit family $V$ of rational curves, is there an extremal elementary contraction which contracts
all curves in the family or, in other words, does the numerical class of a  curve
in the family span an extremal ray of $\cone(X)$?
\end{minipage}
\end{center}\par
\bigskip
As proved in  \cite{Cam} (see also \cite{De} and \cite{Kob}) there is always a rational
fibration, defined on an open set of $X$, whose general fibers are proper, which contracts
a general curve in~$V$.
More precisely, a general fiber is an equivalence class with respect to the  
relation induced by the closure $\V$ of the family $V$ in the Chow scheme of $X$  in the following way:
two points $x$ and $y$ are equivalent if there exists a connected chain of cycles in $\V$
which joins $x$ and $y$.\par
\medskip
By a careful study of this fibration and of its indeterminacy locus, a partial answer to 
this question has been given in \cite[Theorem 2]{BCD}; namely, if the dimension of a general equivalence
class is greater than or equal to the dimension of the variety minus three then the
numerical class of a general curve in the family spans an extremal ray of $\cone(X)$.\par
\medskip
Before the results in \cite{BCD} a special but very natural situation
in which the question arises has been studied in \cite{BSW}. In that paper
manifolds covered by rational curves of degree one with respect
to an ample line bundle $H$ were considered, and it was proved that a covering family of such curves
(we will call them lines, by abuse) of anticanonical degree 
% $-K_X\cdot V$
%(by abuse of notation, throughout the paper we will denote by $L \cdot V$
%the intersection number $L \cdot C_V$, with $C_V$ any curve among those
%parametrized by $V$) 
greater than or equal to $\frac{\dim X + 2}{2}$ spans
an extremal ray (see also \cite[Theorem 2.4]{BI}).\par
\medskip
Recently, in \cite[Theorem 7.3]{NO2}, the extremality of a covering family $V$ of lines was proved
under the weaker assumption that the anticanonical degree of such curves,  denoted by abuse of notation
by $-K_X \cdot V$, is greater than or equal to $\frac{\dim X + 1}{2}$.\par
\medskip

The goal of the present paper is to prove the following
\begin{theorem*}
Let $(X,H)$ be a polarized manifold  with a dominating family
of rational curves $V$ such that $H \cdot V=1$.
If $-K_X \cdot V \ge \frac{\dim X - 1}{2}$, then $[V]$ spans an extremal ray of $\cone(X)$.
\end{theorem*}

The main idea is, as in \cite{NO2}, to combine the ideas and tecniques of \cite{BSW}, especially
taking into consideration a suitable adjoint divisor $K_X +mH$ and studying its nefness,
with those of \cite{BCD}, in particular regarding the existence of special curves in the indeterminacy
locus of the rational fibration associated to $V$.

%%%%%%%%%%%%%%%%%%%%%%%%%%%%%%%%%%%%%%%%%%%%%%%%%%%
%%%%%%%%%%%%%%%%%%%                SECTION 1          %%%%%%%%%%%%%%%%%%%
%%%%%%%%%%%%%%%%%%%%%%%%%%%%%%%%%%%%%%%%%%%%%%%%%%%

\section{Background material}

Let $X$ be a smooth projective variety defined over the field of complex numbers.
A {\em contraction} $\f \colon X \to Z$ is a proper surjective map with connected fibers onto a normal variety $Z$.\\
If the canonical bundle $K_X$ is not nef, then the negative part of
the cone $\cone(X)$ of effective 1-cycles is locally polyhedral, by the Cone Theorem.
By the Contraction Theorem, to every face in this part of the cone  is associated
a contraction.\\
Unless otherwise stated, we will reserve the name {\em extremal face} for a face contained in $\overline{\cone}(X) \cap \{a \in \cycl(X)\;|\; K_X \cdot a <0\}$, and we will call {\em extremal contraction}
the contraction of such a face.\\
An extremal contraction associated to an extremal face of dimension one,
{\em i.e.} to an extremal ray, is called an {\em elementary
contraction}; an extremal ray $\tau$ is called {\em numerically effective}, and the
associated contraction is said to be of {\em fiber type}, if $\dim Z < \dim X$;
otherwise the ray is called {\em non nef} and the contraction is {\em birational}.\\
If the codimension of the exceptional locus of an elementary birational contraction
is equal to one, the ray and the contraction are called {\em divisorial}, otherwise they are 
called {\em small}.\\
A Cartier divisor which is the pull-back of an ample divisor $A$ on $Z$ is called a {\em supporting divisor}
of the contraction $\f$.\\
If the anticanonical bundle of $X$ is ample, $X$ is called a Fano manifold.
For a Fano manifold, the {\em index}, denoted by $r_X$, is defined as the largest natural number $r$ 
such that $-K_X=rH$ for some (ample) divisor $H$ on $X$.\\
Throughout the paper, unless otherwise stated, we will use the word {\em curve} to denote
an irreducible  curve.

\begin{definition} \label{Rf}
A {\em family of rational curves} is an irreducible component
$V \subset \ratcurves$ (see \cite[Definition 2.11]{Kob}).
Given a rational curve we will call a {\em family of
deformations} of that curve any irreducible component of  $\ratcurves$ 
containing the point parametrizing that curve.
We will say that $V$ is {\em unsplit} if it is proper.
We define $\loc(V)$ to be the set of points of $X$ through which there is a curve among those
parametrized by $V$; we say that $V$ is a {\em covering family} if ${\loc(V)}=X$ and that $V$ is a 
{\em dominating family} if $\overline{\loc(V)}=X$.\\
We denote by $V_x$ the subscheme of $V$ parametrizing rational curves 
passing through $x \in \loc(V)$ and by $\loc(V_x)$ the set of points of $X$ 
through which there is a curve among those parametrized by $V_x$.\\
By abuse of notation, given a line bundle $L \in \pic(X)$, we will denote by $L \cdot V$
the intersection number $L \cdot C_V$, with $C_V$ any curve among those
parametrized by $V$.
\end{definition}

\begin{proposition}{\rm (}\cite[IV.2.6]{Kob}{\rm )}\label{iowifam}
Let $V$ be an unsplit family of rational curves on $X$. Then
  \begin{itemize}
       \item[(a)] $\dim \loc(V)+\dim \loc(V_x) \ge \dim X  -K_X \cdot V -1$;
       \item[(b)] every irreducible component of $\loc(V_x)$ has dimension $\ge -K_X \cdot V -1$.
    \end{itemize}
\end{proposition}

This last proposition, in case $V$ is the unsplit family of deformations of a rational curve of minimal 
anticanonical degree in an extremal face of $\cone(X)$, gives the {\em fiber locus inequality}:

\begin{proposition}{\rm (}\cite[Theorem 0.4]{Io}, \cite[Theorem 1.1]{Wicon}{\rm )}\label{fiberlocus}
Let $\f$ be a Fano--Mori contraction
of $X$. Denote by $E$ the exceptional locus of $\f$ and
by $F$ an irreducible component of a non-trivial fiber of $\f$. Then
$$\dim E + \dim F \geq \dim X + \ell -1,$$
where $\ell :=  \min \{ -K_X \cdot C\ |\  C \textrm{~is a rational curve in~} F\}$.
If $\f$ is the contraction of an extremal ray $\tau$, then $\ell(\tau):=\ell$ is called the {\em length of the ray}.
\end{proposition}

\begin{definition}\label{CF}
We define a {\em Chow family of rational curves} $\W$ to be an irreducible
component of  $\textrm{Chow}(X)$ parametrizing rational and connected 1-cycles.\\
We define $\loc(\W)$ to be the set of points of $X$ through which there is a cycle among those
parametrized by $\W$; notice that $\loc(\W)$ is a closed subset of $X$ (\cite[II.2.3]{Kob}).
We say that $\W$ is a {\em covering family} if $\loc(\W)=X$.
\end{definition}

\begin{definition}
If $V$ is a family of rational curves, the closure of the image of
$V$ in $\textrm{Chow}(X)$, denoted by $\V$, is called the {\em Chow family associated to} $V$.
\end{definition}

\begin{remark}
If $V$ is proper, {\em i.e.} if the family is unsplit, then $V$ corresponds to the normalization
of the associated Chow family $\V$.
\end{remark}

\begin{definition}
Let $\V$ be the Chow family associated to a family of rational curves $V$. We say that
$V$ (and also $\V$) is {\em quasi-unsplit} if every component of any reducible cycle in $\V$ is
numerically proportional to $V$.
\end{definition}

\begin{definition}
Let $\W$ be a Chow family of rational curves
on $X$ and $Z \subset X$. We define $\loc(\W)_Z$ to be the set of points $x \in X$ such that there exists
a cycle $\Gamma$ among those parametrized by $\W$ with $\Gamma \cap Z \not = \emptyset$ and $x \in \Gamma$.\\
We define $\cloc(\W)_Z$ to be the set of points $x \in X$ such that there exists
a chain of cycles among those parametrized by $\W$ connecting $x$ and $Z$.
Notice that, a priori $\cloc(\W)_Z$ is a countable union of closed subsets of $X$. \end{definition}

{\bf Notation}: If $T \subset X$ we will denote by $\cycx{T} \subset \cycl(X)$ the vector subspace
generated by numerical classes of curves in $T$; we will denote by  $\conx{T} \subset \cone(X)$
the subcone generated by numerical classes of curves in $T$.\\
The notation $\langle \dots \rangle$ will denote a linear subspace, while the notation
 $\langle \dots \rangle_c$ will denote a subcone.

\begin{lemma}\label{numeq}{\rm (}\cite[Proposition IV.3.13.3]{Kob}, \cite[Lemma 4.1]{ACO}{\rm )}
Let $T \subset X$ be a closed subset and let $\W$ be a Chow family of rational curves. Then every curve
contained in $\cloc(\W)_T$ is numerically equivalent to a linear combination with rational
coefficients of a curve contained in $T$ and irreducible components of cycles among those parametrized
by $\W$ which intersect $T$.
%In particular, if $V$ is a quasi-unsplit family, every curve
%contained in $\cloc(\V)_T$ is numerically equivalent to a linear combination with rational
%coefficients of a curve contained in $T$ and a curve among those parametrized by $V$.
%We will write 
%$$\cycx{\cloc(\V)_T} = \langle \conx{T}, [V] \rangle.$$
\end{lemma}

\begin{lemma} {\rm(Cf.} \cite[Proof of Lemma 1.4.5]{BSW}, \cite[Lemma 1]{op}{\rm)} \label{numequns}
Let $T \subset X$ be a closed subset and let $V$ be a quasi-unsplit family of rational curves. Then
every curve contained in $\cloc(\V)_T$ is numerically equivalent to
a linear combination with rational coefficients
   $$\lambda C_T + \mu C_V,$$
where $C_T$ is a curve in $T$, $C_V$ is a curve among those parametrized by $V$ and $\lambda \ge 0$.
\end{lemma}

\begin{corollary} {\rm(Cf.} \cite[Corollary 2.2 and Remark 2.4]{CO}{\rm)}\label{locvf}
Let $\Sigma$ be an extremal face of $\cone(X)$ and denote by $F$ a fiber
of the contraction associated to $\Sigma$. Let $V$ be a quasi-unsplit family numerically independent from curves
whose numerical class is in $\Sigma$. Then
$$\conx{\cloc(\V)_F} = \langle \Sigma, [V] \rangle_c,$$
{\em i.e.} the numerical class in $X$ of a curve in $\cloc(\V)_F$
is in the subcone of $\cone(X)$ generated by $\Sigma$ and $[V]$.
\end{corollary}

\begin{lemma}\label{efnef} Let  $D$ be an effective  divisor on $X$
and $L$ a nef divisor. If $(L+D)|_{D}$ is nef then $L+D$ is nef. 
\end{lemma}

\begin{proof} Assume that $\gamma$ is an effective curve on $X$ such that
$(L+D) \cdot \gamma < 0$. By the nefness of $L$ we have $D \cdot \gamma <0$,
hence $\gamma \subset D$. But $L+D$ is nef on $D$, a contradiction.
\end{proof}

%%%%%%%%%%%%%%%%%%%%%%%%%%%%%%%%%%%%%%%%%%%%%%%%%%%
%%%%%%%%%%%%%%%%%%%                SECTION 2             %%%%%%%%%%%%%%%%%%
%%%%%%%%%%%%%%%%%%%%%%%%%%%%%%%%%%%%%%%%%%%%%%%%%%%

\section{Rationally connected fibrations}

Let $X$ be a smooth complex projective variety and let $\W$ be
a covering Chow family of rational curves.

\begin{definition}
The family $\W$ defines a relation of rational connectedness with respect
to $\W$, which we shall call {\em rc$(\W)$-relation} for short, in the following
way: $x$ and $y$ are in rc$(\W)$-relation if there exists a chain of cycles among those
parametrized by $\W$ which joins $x$ and $y$.
\end{definition}

To the rc$(\W)$-relation we can associate a fibration, at least on an open subset
(\cite{Cam81}, \cite[IV.4.16]{Kob}); we will call it {\em rc$(\W)$-fibration}.

%We have a diagram, coming from the universal family over $\textrm{Chow}(X)$:
%\begin{equation}\label{chowdiagram}
%$$\xymatrix@=30pt{\mathcal U  \ar[r]^{i} \ar[d]_{p} & X\\
 %\W. &  }$$
%\end{equation}
%In the above diagram, the map $i$ is induced by the evaluation and the fibers of $p$ are connected with
%rational components. Moreover, both $i$ and $p$ are proper (see for instance \cite[II.2.2]{Kob}). \\
In the notation of \cite{BCD}, by \cite[Theorem 5.9]{De}  there exists
a closed irreducible subset of $\Ch(X)$ such that, denoting by $Y$ its normalization and by
$Z \subset Y \times X$ the restriction of the universal family, we have
a commutative diagram
\begin{equation}\label{diagram}
\xymatrix@=30pt{Z \ar[r]^e \ar[d]_p & X \ar@{-->}[ld]^q \\
Y &}
\end{equation}
where $p$ is the projection onto the first factor and  $e$ is a birational morphism whose exceptional locus $E$ does not dominate $Y$.
Moreover, a general fiber of $q$ is irreducible and is a rc$(\W)$-equivalence class. \\
Let $B$ be the image of $E$ in $X$;
note that $\dim B \le \dim X-2$, as $X$ is smooth. \par
\medskip
If we consider a (covering) Chow family $\V$, associated to a quasi-unsplit dominating family $V$,
then by \cite[Proposition 1, (ii)]{BCD} $B$ is the union of all rc$(\V)$-equivalence classes
of dimension greater than $\dim X - \dim Y$.

Moreover we have the following

\setcounter{equation}{0}

%Set $X^0:=X \setminus B$ and $Y^0:=Y \setminus p(E)$; choose a very ample
%divisor $D$ on $Y^0$ and set $\widehat D:=\overline{q^{-1}D}$.\\
%We now summarize some properties of this divisor.

\begin{lemma}\label{inB}
Let $V$ be a quasi-unsplit dominating family of rational curves on a
smooth complex projective variety $X$.
Let $B$ be the indeterminacy locus of the rc$(\V)$-fibration $q\colon \rc{X}{Y}$,
let $D$ be a very ample divisor on $q(X \setminus B)$ and let $\widehat D:=\overline{q^{-1}D}$. Then
\begin{enumerate}
\item $\widehat D \cdot V=0$;
\item if $C \not \subset B$ is a curve not numerically  proportional to $[V]$,
then $\Dcap \cdot C >0$;
\item if $\widehat D \cdot C > 0$ for every curve $C \subset B$ not numerically  proportional to $[V]$, then $[V]$ spans an extremal ray of $\cone(X)$.
\end{enumerate}
\end{lemma}

\begin{proof}
See \cite[Proof of Proposition 1]{BCD}.
\end{proof}

\begin{corollary}\cite[Proposition 3]{BCD}. \label{incompreso}
Let $V$ be a quasi-unsplit dominating family of rational curves on a
smooth complex projective variety $X$; denote by $B$ the indeterminacy locus of the rc$(\V)$-fibration
and by $f_V$ the dimension of the general rc$(\V)$-equivalence class.\\
%$q\colon \rc{X}{Y}$.
If $[V]$ does not span an extremal ray of $\cone(X)$, then $B$ is not empty. In particular there exist
rc$(\V)$-equivalence classes of dimension $\ge f_V +1$.
\end{corollary}

We now give a lower bound on the dimension of $\cloc(\V)_S$, depending on the position
of the subvariety $S$ with respect to the indeterminacy locus of the
rc$(\V)$-fibration.

\begin{lemma}\label{primaopoi} Let $V$ be a quasi-unsplit dominating family of rational curves on a
smooth complex projective variety $X$; denote by $B$ the indeterminacy locus of the rc$(\V)$-fibration
and by $f_V$ the dimension of the general rc$(\V)$-equivalence class.\\
Let $S \subset X$ be an irreducible subvariety such that $[V] \not \in \conx{S}$. 
Then there exists an irreducible  $X_S$ contained in
$\cloc(\V)_S$ such that
\begin{enumerate}
\item[(1)] if $S \not\subset B$, then $\dim X_S \ge \dim S + f_V$;
\item[(2)] if $S \subset B$, then $\dim X_S \ge \dim S + f_V +1$.
\end{enumerate}
Moreover, $X_S$  is not rc$(\V)$-connected.
\end{lemma}

\begin{proof} We refer to diagram (\ref{diagram}). Given any $T \subset Z$ we will set $Z_T:=p^{-1}(p(T))$. Let $S' \subset Z$ be an irreducible
component of $e^{-1}(S)$ which dominates $S$ via $e$.\\
By our assumptions on $\conx{S}$ we have that $S'$ meets any fiber of $p|_{Z_{S'}}$   in points so, up to replace $Z_{S'}$ with $S' \times_{p(S')} Z_{S'}$, we can assume that $S'$ is a section of $p|_{Z_{S'}}$.\\
Let $Z'$ be an irreducible  component of $Z_{S'}$ which contains $S'$. We have 
\begin{equation}\label{zetaesse}
\dim Z' \ge  \dim p(S') +f_V = \dim S' + f_V \geq \dim S +f_V.
\end{equation}
Moreover, notice that $S = e(S') \subset e(Z') \subset e(Z_{S'}) \subset \cloc(\V)_S$.\par
\medskip
Assume that $S \not \subset B$. Then $Z' \not \subset E$, hence the map $e|_{Z'} \colon Z' \to X$ is generically finite.
Therefore, in view of (\ref{zetaesse}), $\dim e(Z') = \dim Z' \geq \dim S + f_V$; moreover, since $S \subset e(Z')$ we have that $e(Z')$ is not rc$(\V)$-connected.\par
\medskip
Assume now that $S \subset B$. % As $e(Z_S) \subset \cloc(\V)_S$ and $B$ is closed
%with respect to rc$(\V)$-equivalence, we have $e(Z_S) \subset B$.\\
Assertion (2) will follow once we prove that the general fiber $G$ of $e|_{\oz}$
has dimension strictly smaller than the general fiber of $e|_{S'}$ for at least one irreducible component $\oz$ of $Z_{S'}$ which dominates $p(S')$.
In fact, recalling also (\ref{zetaesse}), in this case we will have
$$\dim e(\oz) = \dim \oz - \dim G > (\dim S' +f_V) -(\dim S'-\dim S) =f_V +\dim S.$$

\textbf{Claim.}  Let $G$ be an irreducible component of a fiber of  $e|_{Z_{S'}}$,  let $z \in G$ be any point and let $z':=p^{-1}(p(z)) \cap S'$ be the intersection of the fiber of $p$ containing $z$ with $S'$; then there exists an irreducible component $F$ of the fiber $F'$ of $e|_{S'}$ containing $z'$
such that $p(G) \subseteq p(F)$.\par
\medskip
To prove the claim, recall that, since $e(Z_G) \subset \cloc(\V)_{e(z)}$, the image via $e$ of any curve in $Z_G \cap S'$ -- which is irreducible, being a section over $p(G)$ -- must be a point,
otherwise it would be a curve contained in $S \cap \cloc(\V)_{e(z)}$, which is a contradiction, since
curves in $S$ are numerically independent from $[V]$.\\
Therefore $Z_G \cap S'$ is contained in a fiber $F'$ of $e|_{S'}$. To prove the claim we take as $F$ the irreducible component
of $F'$ containing $Z_G \cap S'$.\par
\medskip
Let $S^1 \subset S'$ be the proper closed subset on which $e|_{S'}$ is not equidimensional and let $S^2 \subset S'$ be the proper closed subset of points in which  the fiber of $e|_{S'}$ is not locally irreducible.
Recalling that $p|_{S'}$ is a finite map we see that $p\,(S^1 \cup S^2)$ is a proper closed subset of
$p(S')$.\\
Let $y \in p(S') \setminus p\,(S^1 \cup S^2)$ be a general point; in particular  there is only one irreducible component $F$ of the fiber $F'$ of $e|_{S'}$ passing through $z'= p^{-1}(y) \cap S'$ and $\dim F = \dim S' -\dim S$.\\
Notice that $\dim e(Z_F) > f_V$, otherwise a one parameter family of fibers of $p$ meeting $F$ would have the same image in $X$
 (Cf. \cite[End of proof of Proposition 1]{BCD}, where $e(Z_F)=\loc(V_{e(F)} )$).\\
 This implies that, for an irreducible component $\overline Z_F$ of $Z_F$ we have
 $\dim e(\overline{Z}_F) > f_V$.
 Taking as $\oz$ an irreducible component of $Z_{S'}$ containing $\overline{Z}_F$ we have that, for every point $z \in p^{-1}(y) \cap \oz$ and any irreducible component $G$ of  the fiber of $e|_{\oz}$ passing through $z$ we have $p(G) \subseteq p(F)$, hence
$\dim G < \dim F=\dim S' -\dim S$; the same inequality then holds for the general fiber by semicontinuity of the local dimension.\\
Noticing that $S$ is contained in $\cloc_{e(\oz)}(\V)$ the last assertion follows.
\end{proof}

\begin{remark} Both the bounds in Lemma (\ref{primaopoi}) are sharp. An example for the
second one is given by \cite[Example 2]{BCD}: in that example $B \simeq \pd \times \pu$;
taking as $S$ a fiber of the projection onto $\pd$ we have equality in $(2)$.
\end{remark}

%%%%%%%%%%%%%%%%%%%%%%%%%%%%%%%%%%%%%%%%%%%%%%%%%%%
%%%%%%%%%%%%%%%%%%%                SECTION 3             %%%%%%%%%%%%%%%%%%
%%%%%%%%%%%%%%%%%%%%%%%%%%%%%%%%%%%%%%%%%%%%%%%%%%%

\section{Blowing-down}

In this section we consider the following situation, which will show up in the
proof of Theorem (\ref{main}):

\begin{lemma}\label{blow}
Let $(X,H)$ be a polarized manifold  with a dominating family of rational curves $V$ such that $H \cdot V=1$.
Denote by $f_V$ the dimension of the general rc$(\V)$-equivalence class and assume that
there exists an  extremal face $\Sigma$ in $\cone(X)$ whose associate contraction
$\sigma \colon X \to X'$ is a smooth blow-up along a disjoint union of subvarieties $T_i$
of dimension $\le f_V$ such that $E_i \cdot V=0$ for every exceptional divisor~$E_i$ and $H \cdot l_i=1$ if $l_i$ is a line in a fiber of $\sigma$.
Finally denote by $V'$  a family of deformation of $\sigma(C)$, with $C$ a general curve parametrized by $V$.
Then
\begin{enumerate}
\item $-K_{X'} \cdot V' =-K_X \cdot V$;
\item there exists an ample line bundle $H'$ on $X'$ such that $H' \cdot V' =1$;
\item if $C'$ is a curve parametrized by $V'$ such that $T_i \cap C'\not=\emptyset$, then $C' \subset T_i$;
\item $\rho_{X'} > 1$;
\item  if $[V']$ spans an extremal ray of $\cone(X')$, then $[V]$ spans an extremal ray of $\cone(X)$.
\end{enumerate}
\end{lemma}

\begin{proof}
It is enough to prove the statement in case $\dim \Sigma =1$, {\em
i.e.} $\sigma\colon X \to X'$ is the blow-up of $X'$ along a smooth
subvariety $T$ associated to the extremal ray $\Sigma$. In fact,
if $\dim \Sigma >1$, the contraction of $\Sigma$ factors through
elementary contractions, each one satisfying the assumptions in
the statement.\par
\medskip
Denote by $E$ the exceptional locus of $\sigma$. Since $E \cdot V=0$ the first
assertion in the statement follows from the canonical bundle formula for blow-ups.\par
\medskip
Moreover, the fact that $E \cdot V=0$ also implies that any rc$(\V)$-equivalence class meeting
$E$ is actually contained in $E$.
Therefore, if $F$ is a non-trivial fiber of $\sigma$, then $\cloc(\V)_F \subseteq E$.
By Lemma (\ref{primaopoi})
$$\dim \cloc(\V)_F \ge f_V + \dim F \ge \dim X-1,$$
hence $E= \cloc(\V)_F$ and $\dim T=f_V$; in particular, applying  Corollary (\ref{locvf}) we get  that
 $\conx{E}=\langle [V], \Sigma \rangle_c$.\par
 \medskip
The line bundle $(H +E)|_{E}$ is nef and it is trivial only on $\Sigma$,
since  $(H +E) \cdot \Sigma=0$ and $(H +E) \cdot V=1$. Then $H+E$ is nef
by Lemma (\ref{efnef}).\\
Notice also that $H+E$ is trivial only on $\Sigma$.
Indeed, let $\gamma$ be an effective curve on $X$ such that $(H+E) \cdot \gamma = 0$.
Due to the ampleness of $H$ we have $E \cdot \gamma <0$, hence $\gamma \subset E$.
This implies that $[\gamma] \in \Sigma$.
Therefore $H+E = \sigma^\ast H'$, with $H'$ an ample line bundle on $X'$.
By the projection formula $H' \cdot V' =1$, hence part (2) in the statement is proved.\par
\medskip
Now, let $C'$ be a curve parametrized by $V'$ meeting $T$ and assume by contradiction that $C'$ is not contained in $T$;
denote by $\tl C'$ its strict transform. Then
$$1 = H' \cdot C' = \sigma^\ast H' \cdot \tl C' = (H+E) \cdot \tl C' \geq 2,$$
which is a contradiction.
It follows that every curve parametrized by $V'$ which meets $T$ is contained in it;
so we get part (3) in the statement.\par
\medskip
As to part (4), assume by contradiction that $\rho_{X'} =1$. This implies that
$X'$ is rc$(\V')$-connected, but this is impossible as,
in view of part (3),
we cannot join points of $T$ and points outside of $T$ with curves parametrized by $V'$.\par
\medskip
Finally, to prove part (5) assume  that $[V']$ spans an extremal ray of $X'$ and let
$B$ be the indeterminacy locus of the rc$(\V)$-fibration. We claim that $E \cap B = \emptyset$.\par
\smallskip
Assume by contradiction that this is not the case; then
$E$ meets (and hence contains) an rc$(\V)$-equivalence class $G$ of dimension $\dim G \ge f_V+1$.
Take a fiber $F$ of $\sigma$ meeting $G$. Then
$\dim F + \dim G > \dim E$. On the other hand, $\dim (F \cap G)=0$ as $[V] \not \in \Sigma$.
So we get a contradiction.\par
\smallskip
Let $A$ be a supporting divisor
of the contraction associated to $[V']$. The pull-back $\sigma^\ast A$ defines a two-dimensional
face $\Pi$ of $\overline{\cone}(X)$ containing $\Sigma$ and $[V]$. Let $\hD$ be as in Lemma (\ref{inB}); by the same lemma  $\hD \cdot \Sigma >0$ and $\hD \cdot V=0$.\\
Assume that $\Pi$ is not spanned by $\Sigma$ and $[V]$; in this case
there exists a class $c \in \overline{\cone}(X)$ belonging to $\Pi$ such that $E \cdot c >0$ and $\hD \cdot c <0$.\\
Let $\{C_n\}$ be a sequence of effective one cycles such that the limit of $\R_+[C_n]$ is $\R_+c$; by continuity, for some $n_0$ we have $E \cdot C_{n} >0$ and  $\hD \cdot C_{n} <0$ for $n \ge n_0$,
hence  $C_{n} \subset B$, and $E \cap C_{n} \not = \emptyset$ for $n \ge n_0$, contradicting $E \cap B = \emptyset$.
\end{proof}

%%%%%%%%%%%%%%%%%%%%%%%%%%%%%%%%%%%%%%%%%%%%%%%%%%%
%%%%%%%%%%%%%%%%%%%                SECTION 4            %%%%%%%%%%%%%%%%%%
%%%%%%%%%%%%%%%%%%%%%%%%%%%%%%%%%%%%%%%%%%%%%%%%%%%

\section{Main theorem}
\setcounter{equation}{0}

First of all we consider polarized manifolds $(X,H)$ with a quasi-unsplit dominating
family of rational curves $V$ proving that if, for $m$ large enough, the adjoint divisor
$K_X+mH$ defines an extremal face containing $[V]$, then $[V]$ spans an extremal
ray of $X$.

\begin{proposition}\label{ivan} Let $(X,H)$ be a polarized manifold
which admits a quasi-unsplit dominating family of
rational curves $V$; denote by $f_V$ the dimension of a general rc$(\V)$-equivalence class.\\
If, for some integer $m$ such that $m +f_V \ge \dim X-3$, the divisor
$K_X +mH$ is nef and it is trivial on $[V]$, then $[V]$ spans an extremal ray of $\cone(X)$.
\end{proposition}

\begin{proof}
Assume by contradiction that $[V]$ does not span an extremal ray in $\cone(X)$.\\
This implies that $K_X+mH$ defines an extremal face $\Sigma$ of dimension at least two, containing $[V]$.
By \cite[Lemma 7.2]{NO2} there exists an extremal ray $\vt \in \Sigma$
whose exceptional locus is contained in the indeterminacy locus $B$ of the rc$(\V)$-fibration.
Since $(K_X +mH)\cdot \vt = 0$, the length $\ell(\vt)$ is greater than or equal to $m$.\\
Let $F$ be a non-trivial fiber of the  contraction associated to $\vt$; since this contraction  is small, being $\dim B \le \dim X-2$,
then $\dim F \ge m+1$ by Proposition (\ref{fiberlocus}).\\
By part $(2)$ of Lemma (\ref{primaopoi}), the dimension of $\cloc(\V)_{F}$ is
$$\dim \cloc(\V)_{F} \ge \dim F + f_V +1.$$
As the rc$(\V)$-equivalence classes are either contained in $B$ or have empty intersection with it,
$\cloc(\V)_{F} \subset B$.
Therefore we get
$$\dim X-2 \ge \dim B \ge \dim \cloc(\V)_{F} \ge f_V + m +2 \ge \dim X-1,$$
which is a contradiction.
\end{proof}

As the last preparatory step, we consider the following special case.

\begin{lemma}\label{premain}
Let $V$ be a quasi-unsplit dominating family of rational curves on a
smooth complex projective variety $X$.
Denote by $f_V$ the dimension of a general rc$(\V)$-equivalence class.
Assume that there exists an extremal ray $\vt$, independent from $[V]$, whose associated contraction
has a fiber $F$ such that $\dim F +f_V \ge \dim X$.
Then $\dim F +f_V = \dim X$ and
$\cone(X)= \langle [V], \vt \rangle_c$. In particular $\rho_X=2$.
\end{lemma}

\begin{proof}
By part (1) of Lemma (\ref{primaopoi}) we have
$$\dim X \ge \dim \cloc(\V)_F \ge f_V + \dim F,$$
hence $\dim F +f_V = \dim X$ and $\cloc(\V)_F =X$;
so the assertion follows by Corollary~(\ref{locvf}).
\end{proof}

\begin{theorem}\label{main}
Let $(X,H)$ be a polarized manifold  with a dominating family
of rational curves $V$ such that $H \cdot V=1$.
If $-K_X \cdot V \ge \frac{\dim X - 1}{2}$, then $[V]$ spans an extremal ray of $\cone(X)$.
\end{theorem}

\begin{proof}
Let $B$ be the indeterminacy locus of the rc$(\V)$-fibration $q\colon \rc{X}{Y}$,
let $D$ be a very ample divisor on $q(X \setminus B)$ and let $\widehat D:=\overline{q^{-1}D}$.
Denote by $m$ the anticanonical degree of $V$ and by
$f_V$ the dimension of a general rc$(\V)$-equivalence class. Notice that, since
$V$ is a dominating family, we have $m \ge 2$.\\
By Proposition (\ref{iowifam}) $\dim \loc(V_x) \ge -K_X \cdot V -1 =m-1$; since
a general fiber of the rc$(\V)$-fibration contains $\loc(V_x)$ for every point $x$ in it,
we have $f_V \ge m-1$.\par
\medskip
If $K_X +mH$ is nef, then the assertion follows by Proposition (\ref{ivan});
therefore we can assume that $K_X+mH$ is not nef.\\
Let $\vt$ be an extremal ray such that $(K_X+mH) \cdot \vt <0$ and let $\f_\vt$ be the associated contraction.
Notice that $\vt$ has length $\ell(\vt)\ge m+1$, hence every non-trivial fiber of
$\f_\vt$ has dimension $\ge m$ by Proposition (\ref{fiberlocus}). On the other hand, by Lemma (\ref{premain})
we can confine to assume that all fibers of $\f_\vt$ have dimension $\le m+1$.\\
In particular this implies that, denoted by $C_\vt$ a minimal degree curve whose numerical class belongs to $\vt$, we have $H \cdot C_\vt =1$. Indeed, if this were not the case, we would have 
$\ell(\vt) \ge 2m +1$, hence every non-trivial fiber of
$\f_\vt$ would have dimension $\ge 2m > m+1$, by Proposition (\ref{fiberlocus}) 
and the fact that $m\geq 2$.\par
\medskip
If the Picard number of $X$ is one the theorem is clearly true, so we can assume that $\rho_X \ge 2$.
Now we split up the proof in two cases, according to the value of $\rho_X$:
first we consider the case $\rho_X=2$ and then the general one.\par
\medskip
{ \bf\em  Case (a)}\quad $\rho_X=2$. \par
\medskip
The proof is based on different arguments, depending on the dimension of the fibers of the contraction associated
to the extremal ray $\vt$.\par
\medskip
{ \bf\em  Case (a1)}\quad The contraction $\f_\vt$ admits an $(m+1)$-dimensional fiber $F$. \par
\medskip
Consider $X_F:=\cloc(\V)_F$. We have, by Corollary (\ref{locvf}), that $\conx{X_F}= \langle [V],\vt\rangle_c$ and,
by Lemma (\ref{primaopoi}), that
$$\dim X_F \ge \dim F + f_V  \ge (m+1)+(m-1)\geq \dim X -1.$$
If $X_F=X$, then the statement is proved. So we can assume that an irreducible component 
$\overline{X}_F$ of $X_F$ is a divisor and thus that $f_V=m-1$.
Notice that $\overline{X}_F \cdot V=0$, otherwise we would have $X_F=X$.\\
Consider now the intersection number of ${X}_F$ with curves whose numerical class belongs to $\vt$; since $\rho_X=2$
and $\overline{X}_F  \cdot V=0$ we cannot have also $\overline{X}_F  \cdot \vt=0$.\par
\smallskip
Let us show that we cannot have $\overline{X}_F  \cdot \vt < 0$, too.\\
Assume by contradiction that this is the case. Then $\Exc(\vt) \subset \overline{X}_F $, so $\f_\vt$ 
is divisorial by Proposition (\ref{fiberlocus}).
By the same proposition, recalling that we are assuming that all the fibers of $\f_\vt$
have dimension $\le m+1$, every non-trivial fiber has dimension $m+1$.\\
Then $\f_\vt$ is the blow-up of a smooth variety $X'$ along a smooth center $T$ by \cite[Theorem 4.1 (iii)]{AWduke}.
The dimension of the center is 
$$\dim T=(n-1) - (m+1) \le m-1=f_V.$$
We can thus apply part (4) of Lemma
(\ref{blow}) and we get $\rho_X = \rho_{X'} +1 >2$,  reaching a contradiction.\par
\smallskip
Therefore $\overline{X}_F \cdot \vt > 0$, hence $(\overline{X}_F)|_{\overline{X}_F}$ is nef and thus,
by Lemma (\ref{efnef}), $\overline{X}_F$ is nef. 
As $\overline{X}_F \cdot V=0$ and $\rho_X=2$, $\overline{X}_F$ is the supporting
divisor of an elementary contraction of $X$ whose associated
extremal ray is spanned by $[V]$.\par
\medskip
{ \bf\em  Case (a2)}\quad The contraction $\f_\vt$ is equidimensional with $m$-dimensional fibers. \par
\medskip
By Proposition (\ref{fiberlocus}), $\f_\vt$ is of fiber type and $\ell(\vt)=m+1$. Hence,
by \cite[Lemma 2.12]{Fuj4}, $X$ is a projective bundle over a
smooth variety $Y$, {\em i.e.}
 $X=\proj_Y(\E)$, where $\E=(\f_\vt)_\ast H$.\\
Notice that $Y$ has Picard number one and is covered by
rational curves -- the images of the curves parametrized by $V$ -- therefore
$Y$ is a Fano manifold.\\
By the canonical bundle formula for projective bundles we have
$$K_X +(m+1)H=\f_\vt^\ast(K_Y +\det \E).$$
In particular, if $C_V$ is a curve among those parametrized by $V$, by the projection formula we can compute
$$(K_Y +\det \E)\cdot (\f_\vt)_\ast(C_V)= (K_X +(m+1)H) \cdot C_V=1.$$
It follows that $(K_Y +\det \E) \cdot \f_\vt(C_V)=1$ and  that $K_Y +\det \E$ is the ample generator of $\pic(Y)$.
The ampleness of $\E$ implies that $\det \E \cdot \f_\vt(C_V) \ge m+1$; therefore $-K_Y \cdot \f_\vt(C_V) \ge m$, hence
the index $r_Y$ of $Y$ is greater than or equal to $m$.\par
\medskip
If $r_Y=m$, denoted by $l$ a rational curve of minimal degree in $Y$,
then $\det \E \cdot l =m+1$; moreover, the splitting type of $\E$, which is ample and of rank $m+1$,
on rational curves of minimal degree is uniform of type $(1,\dots,1)$.\\
We can thus apply \cite[Proposition 1.2]{AWinv}, so we obtain that $X \simeq \proj^m \times Y$.
It follows that the curves of $V$ are contained in the fibers of the first projection
and that $[V]$ spans an extremal ray.\par
\medskip
Therefore we are left with $r_Y \ge m +1$. Recalling that $\dim Y = \dim X - m \le m+1$,
by the Kobayashi--Ochiai Theorem (\cite{KOc})
we get that $Y$ is  a projective space or a hyperquadric.\par
\smallskip

Assume by contradiction that $[V]$ does not span an extremal ray of $X$.\\
By part (3) of Lemma (\ref{inB}) there exists a curve $C \subset B$, whose numerical class is not
proportional to $[V]$, such that $\Dcap \cdot C \le 0$. Actually, since $\rho_X=2$ and $\Dcap \cdot V=0$,
we have $\Dcap \cdot C <0$.\\
%Let $X_C$ be a component of $\cloc(\V)_C$ containing $C$ and of dimension
%$\dim X_C \ge f_V+ \dim C +1 \ge m+1$, whose existence follows from
%part (2) of Lemma (\ref{primaopoi}).\\ 
By part (2) of Lemma (\ref{primaopoi}), there exists  $X_C \subset \cloc(\V)_C$ which is not  rc$(\V)$-connected  such that $\dim X_C \ge f_V+ \dim C +1 \ge m+1$. \\
By Lemma (\ref{numequns}) $\Dcap$ has non positive intersection number with every curve in $X_C$ and
it is trivial only on curves which are numerically proportional to $[V]$.\\
Since $\Dcap \cdot \vt >0$, we have that $\f_\vt$ does not contract curves in $X_C$,
hence $\dim Y \ge \dim X_C \ge m+1$ and so $\dim Y = \dim X_C = m+1$.\\
Since $X_C$ is not  rc$(\V)$-connected, for every point $c$ of $X_C$, the intersection $X_c$ of the rc$(\V)$-equivalence class containing $c$ with $X_C$ has dimension $= m$. In particular $X_C$ is the union of a one parameter family of rc$(\V)$-connected
subvarieties $X_c$.\par
\smallskip
We claim that there exists a line $l$ in $Y$ which is not contained in 
$\f_\vt(X_c)$ for any $c \in C$.
Notice that, since $\f_\vt$ does not contract curves in $X_C$, through a general point $y$ in $Y$
there is a finite number of such subvarieties.\\
If $Y \simeq \proj^{m+1}$, a line joining $y$ with a point outside the union of these
subvarieties has the required property.\\
Assume now that $Y \simeq \quadr^{m+1}$; for any  $y \in  \quadr^{m+1}$ the locus of the lines through
$y$ is a quadric cone $\mathbb {Q}^{m}_y$ with vertex $y$. Therefore, if every line through $y$ is contained in $\f_\vt(X_c)$ for some $c \in C$, then $\mathbb {Q}^{m}_y$ is an irreducible component of  
$\f_\vt(X_c)$; since $X_c$ moves in a one-dimensional family, for the general point $y \in \quadr^{m+1}$, the general line through $y$ has the required property.\par
\smallskip
The splitting type of $\E$ on this line
is one of the following: $(2, 1, \dots, 1)$ if $Y \simeq \mathbb Q^{m+1}$ and
either $(3, 1, \dots, 1)$ or $(2, 2, 1, \dots, 1)$ if $Y \simeq \proj^{m+1}$.
Recalling that $m \ge 2$ we have that, among the summands of $\E_{l}$ there is at
least one $\Ol_\pu(1)$.\\
Consider $\proj_l(\E|_{l})$; its cone of curves is generated by the class
of a line in a fiber of the projection onto $l$ and the class of a minimal section
$C_0$. By the discussion above we have that $H \cdot C_0 = 1$. Moreover,
$\f_\vt^\ast (K_Y+\det \E) \cdot C_0 =1$, hence $[C_0]=[V]$;
in particular $\Dcap$ is nef on $\proj_l(\E|_{l})$.\\
Consider an irreducible curve in $\proj_l(\E|_{l}) \cap X_C$; by our choice of $l$, this
curve is not contained in a rc$(\V)$-equivalence class contained in $X_C$, so
it is negative with respect to
$\Dcap$, a contradiction. The case $\rho_X=2$ is thus completed.\par
\medskip
{ \bf\em  Case (b)}\quad $\rho_X > 2$. \par
\medskip
Notice that, in view of Corollary (\ref{incompreso}), we can confine to
assume that $B \not = \emptyset$; moreover, by 
part (3) of Lemma (\ref{inB}), we can also assume the existence of a curve $C \subset B$
such that $[C]$ is not proportional to $[V]$ and $\Dcap \cdot C \le 0$.\par
\medskip
We claim that $K_X+(m+1)H$ is nef.\\
Assume by contradiction that $K_X+(m+1)H$ is not nef. Let $\tau$ be a ray such that $(K_X+(m+1)H) \cdot \tau <0$,
denote by $C_\tau$ a rational curve of minimal 
anticanonical degree in $\tau$ and by $\f_\tau$ the contraction associated to $\tau$.\\
Notice that $\tau$ has length $\ell(\tau)\ge m+2$, hence every non-trivial fiber of
$\f_\tau$ has dimension $\ge m+1$ by Proposition (\ref{fiberlocus}).\\
On the other hand  $\f_\tau$ cannot have fibers of dimension $>m+1$, otherwise,
by Lemma (\ref{premain}), we would have $\rho_X=2$.
Therefore every non-trivial fiber of $\f_\tau$ has dimension
$m+1$.\\
 In view of Proposition (\ref{fiberlocus}), we thus get that $\f_\tau$ is of fiber type and
that the length of $\tau$ is $\ell(\tau)=m+2$; this last fact gives $H \cdot C_\tau=1$. Let us consider $W_\tau$  to be a minimal degree covering family of curves whose numerical class belongs to $\tau$.\\
Since $B$ is not empty, there are rc$(\V)$-equivalence classes of dimension
$\ge f_V+1 \ge m$; let $G$ be one of these classes. Notice that since $\f_\tau$ is equidimensional
with $(m+1)$-dimensional fibers, we have $f_W=m+1$.
By part (1) of Lemma (\ref{primaopoi}) we have
$$\dim \cloc(\W_\tau)_G \ge \dim G +f_W  =2m +1 \ge \dim X,$$
so by Lemma (\ref{numeq}) we deduce $\rho_X=2$, a contradiction which
proves the nefness of $K_X+(m+1)H$. \par
\medskip
Recall now that the extremal ray $\vartheta$ which we fixed at the beginning of the proof has length $\ell(\vt) \ge m+1$ and is generated by a curve $C_\vartheta$ such that $H \cdot \vartheta=1$,
therefore $(K_X+(m+1)H) \cdot \vartheta=0$ and  $K_X+(m+1)H$ is not ample.\\
Let $\Sigma$ be the extremal face contracted by $K_X+(m+1)H$.
We now consider separately two cases, depending on the existence
in $\Sigma$ of a fiber type extremal ray.\par
\medskip
{ \bf\em  Case (b1)} \quad There exists a fiber type extremal ray $\varrho$ in $\Sigma$.\par
\medskip

Let $\f_\varrho$ be the contraction associated with $\varrho$ and denote by  $W_\varrho$   a minimal degree covering family of curves whose numerical class belongs to $\varrho$.\\
By part (2) of Lemma (\ref{primaopoi}), there exists an irreducible $X_C \subset \cloc(\V)_C$ such that $\dim X_C \ge f_V+2$.\\
According to Lemma (\ref{numequns}), every curve in $X_C$ can be written as
$\alpha [C] +\beta [V]$ with $\alpha \ge 0$; in particular, since $\hD \cdot V=0$ by Lemma (\ref{inB}), 
it follows that $\hD$ is not positive
on any curve contained in $X_C$.
By the same lemma $\hD \cdot W_\varrho >0$, hence $[W_\varrho] \not \in \conx{X_C}$. Therefore
part (1) of Lemma (\ref{primaopoi}) gives
$$\dim \cloc(\W_\varrho)_{X_C} \ge \dim X_C + f_{W_\varrho} \ge f_V+2 + m \ge \dim X,$$
where $f_{W_\varrho}$ is the dimension of the general rc$(\W_\varrho)$-equivalence class.\\
Therefore, by applying twice Lemma (\ref{numequns}), we get that the class of every curve in $X$ can be written as
\begin{equation}\label{cone}
\lambda (\alpha[C] +\beta [V]) +\mu[W_\varrho]
\end{equation}
with $\alpha, \lambda \ge 0$ and $\alpha[C] +\beta [V] \in \conx{X_C}$.\\
This has some very important consequences:
first of all,  since we are assuming $\rho_X >2$, this implies that $\rho_X=3$; in particular $[C]$ is not contained in
the plane $\Pi$ in $\cycl(X)$ spanned by $[W_\varrho]$ and $[V]$.
Moreover  the intersection of $\Pi$ with
$\cone(X)$ is a face of $\cone(X)$.\\
We have to prove that $\Pi \cap \overline{\cone}(X)= \langle [V], [W_\varrho] \rangle_c$.
 If this is not the case, then there exists  a class $a$  such that  $\Pi \cap \overline{\cone}(X) = \langle a, [W_\varrho] \rangle_c$ and $\hD \cdot a <0$.\\
Denote by $b \in \cycl(X)$ a class, not proportional to $[V]$, lying in the intersection of $\partial \overline{\cone}(X)$ with the plane $\Pi'=\cycx{X_C}$ and by $\Pi''$ the plane spanned by $[W_\varrho]$ and $b$.\\
Formula (\ref{cone}), traslated in geometric terms, says that ${\cone(X)}$ is contained in the intersection of  half-spaces determined by $\Pi$ and by $\Pi''$ as in the figure below, which shows a cross-section of $\overline{\cone}(X)$.
 \begin{center}
\includegraphics[width=7cm]{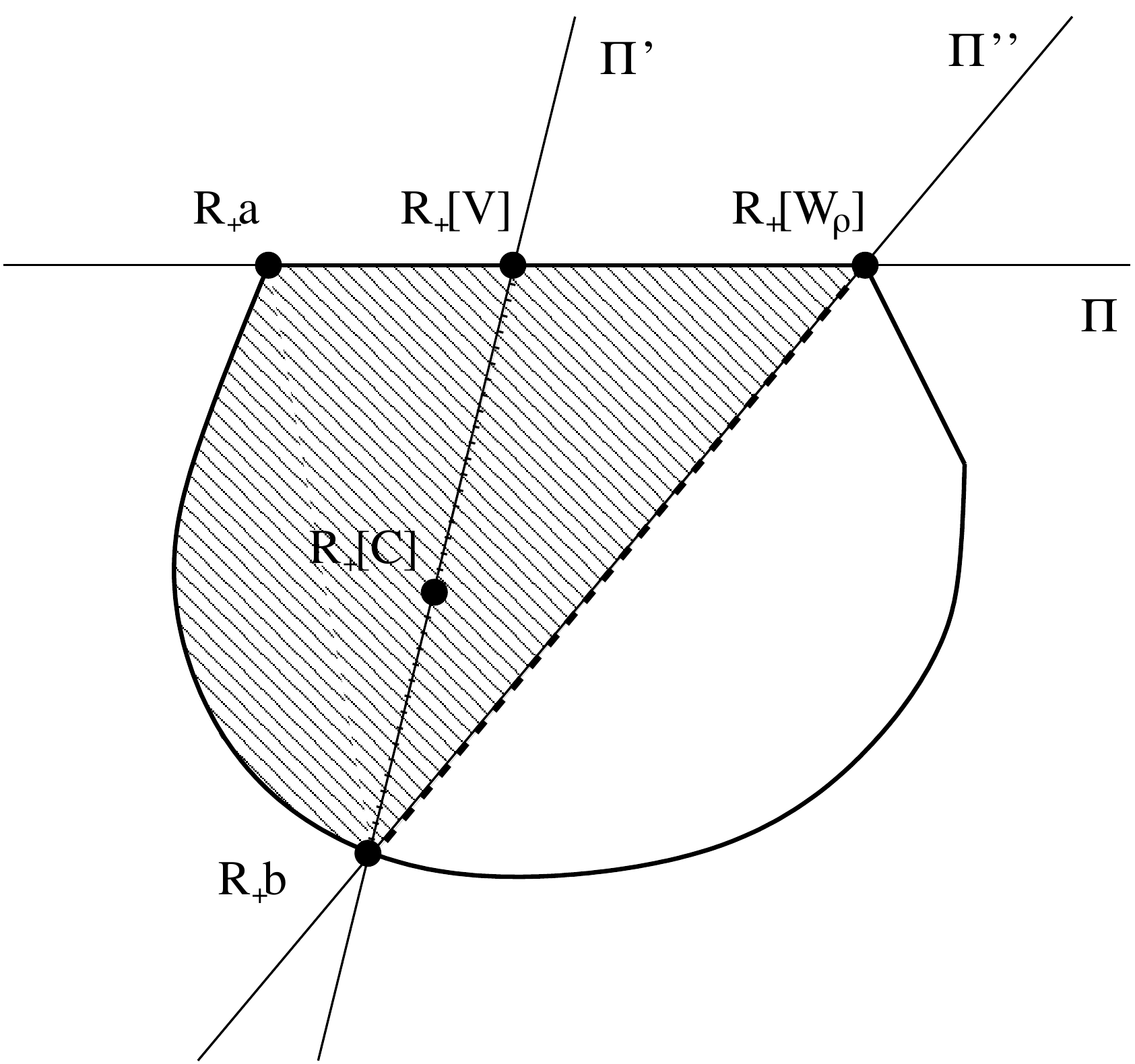}
\end{center}
Let $\{C_n\}$ be a sequence of effective one cycles such that the limit of $\R_+[C_n]$ is $\R_+a$; by continuity, for some $n_0$ we have  $\hD \cdot C_{n} <0$ for $n \ge n_0$, hence  $C_{n} \subset B$ for $n \ge n_0$, and all the above arguments apply to $C_n$, for $n \ge n_0$.
In particular, defining $b_n$ and $\Pi''_n$ as above, we get that, for $n \ge n_0$, ${\cone(X)}$ is contained in the intersection of  half-spaces determined by $\Pi$ and by $\Pi''_{n}$.
Since $\Pi''_n  \rightarrow \Pi$ as  $\R_+[C_n] \rightarrow \R_+a$ and $\rho_X=3$ we get a contradiction.  \par
\medskip
{ \bf\em  Case (b2)} \quad Every ray in $\Sigma$ is birational.\par
\medskip
Let $\eta$ be any ray in $\Sigma$. By Proposition (\ref{fiberlocus}), for every non-trivial fiber of
its associated contraction $\f_\eta$ we have $\dim F \ge \ell(\eta) \ge m+1$. Recalling
that, by Lemma (\ref{premain}), we can assume
$\dim F \le m+1$, we have $\dim F = m+1=\ell(\eta)$. This also implies that, if $C_\eta$ is a minimal degree curve whose numerical class is contained in $\eta$ we have $H \cdot C_\eta=1$.\\
By Proposition (\ref{fiberlocus}), $\f_\eta$ is a divisorial contraction, hence,
by \cite[Theorem 4.1 (iii)]{AWduke}, is the blow-up of a smooth
variety  along a smooth center $T$ of dimension $(n-1)-(m+1) \le m-1$.\par
\medskip
Let $E$ be the exceptional divisor of $\f_\eta$.
By part (2) of Lemma (\ref{primaopoi}), there exists an irreducible $X_C \subset \cloc(\V)_C$ with $\dim X_C \ge f_V+2$.\\
By Lemma (\ref{numequns}) $\Dcap$ has non positive intersection number with every curve in $X_C$.\\
If $E \cap X_C \not = \emptyset$, then there is a fiber $F$ of $\f_\eta$
meeting $X_C$. Counting dimensions, we find that $\dim (F \cap X_C) \ge 1$, which is
a contradiction as $\hD \cdot \eta >0$. So $E \cap X_C = \emptyset$,
whence  $E \cdot V=0$.\\
Therefore $E$ contains  rc$(\V)$-equivalence classes and
$\dim T \ge f_V$, since $\f_\eta$ is finite-to-one on
rc$(\V)$-equivalence classes. Recalling that $f_V \ge m-1$ we derive
$\dim T=f_V=m-1$.\par
\smallskip
Assume that $\dim \Sigma \ge 2$ and let $E_1, E_2$ be the exceptional
loci of two different extremal rays $\eta_1, \eta_2$ in $\Sigma$; since the fibers of the contractions $\f_{\eta_1}$ and $\f_{\eta_2}$ have dimension $m+1$ and $2(m+1) > \dim X$ we have that $E_1 \cap E_2 = \emptyset$.\\
%To prove the claim, notice that $E_i=\cloc(\V)_{F_i}$ for any fiber $F_i$ of $\eta_i$, since $E_i$ is closed with respect to
%rc$(\V)$-equivalence and, by part (1) of Lemma (\ref{primaopoi}), we can compute
%$$\dim \cloc(\V)_{F_i} \ge m+1+f_V = \dim E_i.$$
%So, Corollary (\ref{locvf}) gives $\conx{E_i}=\langle[V], \eta_i\rangle_c$, hence $E_1$ cannot contain a curve
%in a fiber of the contraction associated to $\eta_2$; since the dimension of the fibers is $\ge 2$ this
%implies that $E_1 \cdot R_2 =0$.
%But then, if $E_1 \cap E_2 \not= \emptyset$, then $E_1$ would contain the fibers
%of $\f_2$ which it meets. So we reach a contradiction, proving the claim.\par
%\medskip
Therefore the contraction $\sigma \colon X \to X'$ of the face $\Sigma$ verifies the
assumptions of Lemma (\ref{blow}), hence there exists an ample line bundle
$H'$ on $X'$ and an unsplit dominating family $V'$ on $X'$ such that $H' \cdot V'=1$ and
$-K_{X'} \cdot V'=-K_X \cdot V \ge \frac{\dim X'-1}{2}$.\\
Denote by $f_{V'}$ the dimension of the general rc$(\V')$-equivalence class.
Since a general fiber of the rc$(\V')$-fibration contains $\loc (V'_{x'})$, we have
$f_{V'}\geq \dim \loc (V'_{x'})-1\geq m-1$.\\
Consider the adjoint divisor $K_{X'}+ mH'$; if it is nef,
or an extremal ray $\vt'$ such that $(K_{X'}+ mH') \cdot \vt' <0$ has a fiber of dimension $\ge m+2$,
then $[V']$ spans an extremal ray by Proposition (\ref{ivan}) or by Lemma (\ref{premain}),
so $[V]$ spans an extremal ray by Lemma (\ref{blow}).\par
\medskip
Let us show that the remaining case does not happen.\\
Assume that there is an extremal ray $\vt'$ such that $(K_{X'}+ mH') \cdot \vt' <0$ and
every fiber of the associated contraction has dimension $\le m+1$. In particular we have
$H' \cdot \vt' =1$, otherwise we would have $\ell(\vt') \ge 2m +1$, hence every non-trivial fiber 
of the associated contraction would have dimension $\ge 2m > m+1$ by Proposition
(\ref{fiberlocus}).
Moreover, we have $(K_{X'}+(m+1)H') \cdot \vt' \le 0$, since $\ell(\vt') \ge m+1$.\\
On the other hand, recalling that $\sigma^\ast H'= H+ \sum E_i$ and that 
$\sigma^\ast K_{X'}=K_X -\sum (m+1)E_i$, we have
$$\sigma^\ast (K_{X'}+(m+1)H')=K_{X}+(m+1)H,$$
so, by the projection formula, $K_{X'}+(m+1)H'$ is ample on $X'$, a contradiction.
\end{proof}

\begin{corollary} Let $(X,H)$ be a polarized manifold of dimension at most five, with
a dominating family of rational curves $V$ such that $H \cdot V=1$.
Then $[V]$ spans an extremal ray of $\cone(X)$.
\end{corollary}

%%%%%%%%%%%%%%%%%%%%%%%%%%%%%%%%%%%%%%%%%%%%%%%%%%%
%%%%%%%%%%%%%%%%%%%                SECTION 5             %%%%%%%%%%%%%%%%%%
%%%%%%%%%%%%%%%%%%%%%%%%%%%%%%%%%%%%%%%%%%%%%%%%%%%

\section{An example}

\setcounter{equation}{0}

In the paper \cite{BSW}, an application of the results about extremality of families of
lines was a relative version of a theorem proved in \cite{Wimu}, which
was the first step towards a conjecture of Mukai for Fano manifolds.\\
This conjecture states that, for a Fano manifold $X$, denoted by $\rho_X$  its
Picard number and by $r_X$ its index, we have
$$\rho_X(r_X-1) \le \dim X.$$
More precisely, in \cite[Theorem B]{Wimu} it was proved that, if $r_X \ge \frac{\dim X}{2}+1$, then
$\rho_X= 1$ unless $X \simeq \proj^{\dim X/2} \times \proj^{\dim X/2}$, while 
in \cite[Theorem 3.1.1]{BSW} it was proved that
a fiber type contraction $\f\colon X \to Y$ supported by $K_X + mL$
with $m \ge \frac{\dim X}{2}+1$ is elementary,
unless $X \simeq \proj^{\dim X/2} \times \proj^{\dim X/2}$.\par
In the last few years some progress has been made towards Mukai conjecture; in
particular it was recently proved in \cite[Theorem 3]{NO} that it holds for a Fano manifold
with (pseudo)index greater than or equal to $\frac{\dim X}{3}+1$.\\
It is therefore natural to ask if the corresponding relative statement is true,
namely, given a fiber type contraction $\f \colon X \to Y$, corresponding to an extremal
face $\Sigma$, supported by $K_X + mL$ with $m \ge \frac{\dim X}{3}+1$
is it possible to find a bound on the dimension of~$\Sigma$?\par
The answer to this question is negative, as we will show with an example in which
$m=\frac{\dim X}{2}$; it follows that \cite[Theorem 3.1.1]{BSW} cannot be improved.

\begin{example}
Let $Z$ be a smooth variety of dimension $k+2$, denote by $Y$ the product
$Z \times \proj^k$ and by $p_1, p_2$ the projections onto the factors.
Let $\{z_i\}_{i=1, \dots, t}$ be points of $Z$ and denote by $F_i$ the fibers of $p_1$ over $z_i$.\\
Let $\sigma\colon X \to Y$ be the blow-up of $Y$ along the union of $F_i$.
The canonical bundle of $X$ is
\begin{equation}\label{serieA}
K_X= \sigma^\ast K_Y + (k+1)\sum_{i=1}^t E_i=\sigma^\ast(p_1^\ast K_Z +p_2^\ast K_{\proj^k}) + (k+1)\sum_{i=1}^t E_i;
\end{equation}
denoting by $\ac:= (p_2 \circ \sigma)^\ast\Ol_{\proj^k}(1)$
and by $L':= \ac -\sum E_i$, we can rewrite formula $(\ref{serieA})$
as
$$K_X+(k+1)L'=\sigma^\ast(p_1^\ast K_Z).$$
It is easy to check that $L'$ is $(p_1 \circ \sigma)$-ample. Let
$A \in \pic(Z)$ be an ample line bundle such that $K_Z+ (k+1)A$ is ample; then $L:=L'+\sigma^\ast (p_1^\ast A)$ is an
ample line bundle on $X$; moreover $L \cdot l=1$ for a line $l$ in the strict transform of a fiber $F$
of $p_1$ not contained in the center of~$\sigma$.\\
The contraction $p_1 \circ \sigma$ is supported by $K_X +(k+1)L=K_X + \frac{\dim X}{2}L$ and contracts a face of
dimension $t+1$.
\end{example}

\begin{remark} The difference between the relative and the absolute case is given by the existence
of minimal horizontal dominating families of rational curves for proper morphisms defined on a
open subset of a Fano manifold (for the definition and the references see
\cite[Remark 6.4]{ACO}). Such families do not exist in general in the relative case.
\end{remark}

\section*{Acknowledgements}

We learned of the results about extremality of families of lines in \cite{BSW}
from an interesting series of lectures given by Paltin Ionescu. We thank the referee for many 
useful suggestions and remarks, which helped to fix some issues in the proofs.

\end{document}